\documentclass{amsart}

\usepackage{amsmath, amsfonts, mathrsfs}
\usepackage{amssymb,latexsym}
\usepackage{enumerate}
\usepackage{bbm} 
\usepackage{pdfpages}

\usepackage[top=2.1cm, bottom=2.1cm, left=2.1cm, right=2.1cm, includefoot, heightrounded]{geometry}

\newtheorem{theorem}{Theorem}[section]

\newtheorem{lemma}[theorem]{Lemma}

\theoremstyle{definition}
\newtheorem{definition}[theorem]{Definition}

\numberwithin{equation}{section}

\begin{document}

\title[Fubini Type Theorems]
{Fubini Type Theorems \\ for the strong McShane and strong Henstock-Kurzweil integrals}

\author{ Sokol Bush Kaliaj }

\address{
Mathematics Department, 
Science Natural Faculty, 
University of Elbasan,
Elbasan, 
Albania.
}

\email{sokol\_bush@yahoo.co.uk}

\thanks{}



\begin{abstract} 
In this paper, we will prove Fubini type theorems for 
the strong McShane and strong Henstock-Kurzweil integrals   
of Banach spaces valued functions defined on a closed non-degenerate interval 
$[a,b] =[a_{1}, b_{1}] \times [a_{2}, b_{2}] \subset \mathbb{R}^{2}$.
\end{abstract}

\maketitle

\section{Introduction and Preliminaries}

The Fubini theorem belongs to the most powerful tools in Analysis. 
It establishes a connection between the so called double integrals and repeated integrals. 
Theorem X.2 in \cite{MIK} is the Fubini theorem for Bochner integral 
of Banach spaces valued functions defined on the Cartesian product  $U \times V$  
of Euclidean spaces $U$ and $V$.  
Here, we will prove a Fubini type theorem for the strong McShane integral 
of Banach spaces valued functions defined on two-dimensional compact intervals, see Theorem \ref{T2.1}. 
Henstock in \cite{HEN0} proved a Fubini--Tonelli type theorem for the 
Perron integral of real valued functions defined on two-dimensional compact intervals.
Tuo-Yeong Lee in  \cite{TUOI} and \cite{TUOII} proved several Fubini--Tonelli type theorems for the Henstock--Kurzweil integral of real valued functions defined on $m$-dimensional compact intervals 
in terms of the Henstock variational measures. 
In this paper, we will prove a Fubini type theorem for 
the strong Henstock--Kurzweil integral 
of Banach spaces valued functions defined on two-dimensional compact intervals, see Theorem \ref{T2.2}.

Throughout this paper $X$ denotes a real Banach space with its 
norm $|| \cdot ||$. 
The Euclidean space $\mathbb{R}^{2}$ is equipped with the maximum norm $||\cdot||_{\infty}$. 
$B_{2}(t,r)$ denotes the open ball in $\mathbb{R}^{2}$ with center $t=(t_{1},t_{2}) \in \mathbb{R}^{2}$  
and radius $r > 0$. 
$A^{o}$, $\partial A$ and $|A|$ denote, respectively, 
the interior, boundary and Lebesgue measure of a subset 
$A \subset \mathbb{R}^{2}$. 
For any two vectors $a = (a_{1}, a_{2})$ 
and
$b = (b_{1}, b_{2})$ 
with
$-\infty < a_{i} < b_{i} < +\infty$,  for $i=1, 2$, 
we set
$$
[a,b] = [a_{1}, b_{1}] \times [a_{2}, b_{2}], 
$$  
which is said to be a \textit{closed non-degenerate interval} in $\mathbb{R}^{2}$. 
By $\mathcal{I}_{[a,b]}$ the family  of all closed non-degenerate 
subintervals in $[a, b]$ is denoted.  
It is easy to see that
$$
\mathcal{I}_{[a,b]} = \mathcal{I}_{[a_{1},b_{1}]} \times \mathcal{I}_{[a_{2},b_{2}]}, 
$$
where $\mathcal{I}_{[a_{1},b_{1}]}$ ($\mathcal{I}_{[a_{2},b_{2}]}$) 
is the family  of all closed non-degenerate 
subintervals in $[a_{1},b_{1}]$ ($[a_{2},b_{2}]$).  
A pair $(t, I)$ of an interval $I \in \mathcal{I}_{[a,b]}$ and a point $t \in [a,b]$ 
is called an \textit{$\mathcal{M}$-tagged interval} in $[a,b]$, $t$ is the tag of $I$. 
Requiring $t \in I$ for the tag of $I$ 
we get the concept of an \textit{$\mathcal{HK}$-tagged interval} in $[a,b]$. 
A finite collection $P$ of 
$\mathcal{M}$-tagged intervals ($\mathcal{HK}$-tagged intervals) in $[a,b]$ 
is called an  \textit{$\mathcal{M}$-partition} (\textit{$\mathcal{HK}$-partition}) in $[a,b]$, 
if $\{ I \in \mathcal{I}_{[a,b]} : (t,I) \in P \}$ is a collection of pairwise non-overlapping intervals 
in $\mathcal{I}_{[a,b]}$. 
Two closed non-degenerate intervals $I, J \in \mathcal{I}_{[a,b]}$ are said to be 
\textit{non-overlapping} if $I^{o} \cap J^{o} = \emptyset$.
A positive function $\delta: [a,b] \to (0,+\infty)$ 
is said to be a \textit{gauge} on $[a,b]$. 
We say that an 
$\mathcal{M}$-partition ($\mathcal{HK}$-partition) $P$ 
in $[a,b]$ is 
\begin{itemize}
\item
$\mathcal{M}$-partition ($\mathcal{HK}$-partition) of $[a,b]$, if 
$\bigcup_{(t,I) \in P} I = [a,b]$,
\item 
$\delta$-fine if for each $(t,I) \in P$, we have 
$
I \subset B_{2}(t,\delta(t)). 
$
\end{itemize}
A function $F : \mathcal{I}_{[a,b]} \to X$ 
is said to be an \textit{additive interval function} if
$$
F(I \cup J) = F (I) + F(J)
$$
for any two non-overlapping intervals $I,J \in \mathcal{I}_{[a,b]}$ with 
$I \cup J \in \mathcal{I}_{[a,b]}$.

\begin{definition}\label{SM_SHK}  
A function $f: [a,b] = [a_{1}, b_{1}] \times [a_{2}, b_{2}]   \to X$ is said to be 
\textit{strongly McShane integrable} 
(\textit{strongly Henstock-Kurzweil integrable}) on $[a,b]$,  
if there is an additive interval function $F: \mathcal{I}_{[a,b]} \to X$ 
such that for every $\varepsilon>0$ 
there exists a gauge $\delta$ on $[a,b]$  
such that  
$$
\sum_{(t,I) \in P} \left \| f(t)|I| - F(I) \right \| < \varepsilon
$$
for every $\delta$-fine $\mathcal{M}$-partition ($\mathcal{HK}$-partition)
$P$ of $[a, b]$.

By  Theorem 3.6.5 in \cite{SCH}, if $f$ is 
\textit{strongly McShane integrable} 
(\textit{strongly Henstock-Kurzweil integrable}) on $[a,b]$, 
then $f$ is McShane integrable  
(\textit{Henstock-Kurzweil integrable}) on $[a,b]$ and 
$$
F(I) = (M)\int_{I} f \quad \left (  F(I) = (HK)\int_{I} f \right ), 
\text{ for all }I \in \mathcal{I}_{[a,b]},
$$
where $F$ is the additive interval function from the definition of strong integrability.
\end{definition} 
Let $K$ be a compact non-degenerate interval in an Euclidean space. 
We denote by $\mathcal{SM}(K)$ ($\mathcal{SHK}(K)$)  
the set of all strongly McShane (strongly Henstock-Kurzweil) 
integrable functions defined on $K$ with $X$-values.  
If $\mathcal{M}(K)$ ($\mathcal{HK}(K)$)  
is the set of all McShane (Henstock-Kurzweil) 
integrable functions defined on $K$ with $X$-values, 
then
\begin{equation*}
\mathcal{SHK}(K) \subset \mathcal{HK}(K),~ 
\mathcal{SM}(K) \subset \mathcal{M}(K),~
\mathcal{M}(K) \subset \mathcal{HK}(K),~
\mathcal{SM}(K) \subset \mathcal{SHK}(K).
\end{equation*}
If the Banach space $X$ is finite dimensional, 
then we obtain by Proposition 3.6.6 in \cite{SCH} that 
\begin{equation*}
\mathcal{SHK}(K) = \mathcal{HK}(K) 
\text{ and } 
\mathcal{SM}(K) = \mathcal{M}(K).
\end{equation*}

\begin{definition}\label{SSM_SSHK}
A function $f: [a,b] \to X$ has the property $\mathcal{S^{*}M}$ ($\mathcal{S^{*}HK}$) 
if for every $\varepsilon >0$ there is a gauge $\delta$ on $[a,b]$ such that
\begin{equation*}
\sum_{(t, I) \in P} \sum_{(s, J) \in Q} ||f(t) - f(s)|| . |I \cap J| < \varepsilon
\end{equation*}
for each pair of  $\delta$-fine $\mathcal{M}$-partitions ($\mathcal{HK}$-partitions) 
$P$ and $Q$ of $[a,b]$. 
\end{definition}
We denote by $\mathcal{S^{*}M}(K)$ ($\mathcal{S^{*}HK}(K)$) 
the set of all functions  defined on $K$ with $X$-values 
having the property $\mathcal{S^{*}M}$ ($\mathcal{S^{*}HK}$). 
Clearly, $\mathcal{S^{*}M}(K) \subset \mathcal{S^{*}HK}(K)$. 
By Lemma 3.6.11, Theorem 3.6.13 and Theorem 5.1.4 in \cite{SCH}, we have
\begin{equation*}
\mathcal{S^{*}HK}(K) \subset \mathcal{SHK}(K),~ 
\mathcal{S^{*}M}(K) = \mathcal{SM}(K), ~
 \mathcal{SM}(K) =  \mathcal{B}(K),
\end{equation*} 
where $\mathcal{B}(K)$ is the set of all 
Bochner integrable functions defined on $K$ with $X$-values.

The basic properties of the  McShane and Henstock-Kurzweil integrals 
can be found in \cite{SCH}, \cite{LEE1}, \cite{LEE2}, \cite{TUO}, \cite{BON}, \cite{DIP}, 
\cite{SKV}, 
\cite{CAO}, \cite{FRE}, 
\cite{GORD}, 
\cite{HEN1}-\cite{HEN3} and \cite{KURZ1}-\cite{KURZ3}.

Given a function $f: [a,b] \to X$, 
for each $t_{1} \in [a_{1}, b_{1}]$ and $t_{2} \in [a_{2}, b_{2}]$   
we define $f_{t_{1}}:[a_{2}, b_{2}] \to X$ 
and 
$f_{t_{2}}:[a_{1}, b_{1}] \to X$ 
by setting
$$
f_{t_{1}}(s_{2}) = f(t_{1}, s_{2}), 
\text{ for all }s_{2} \in [a_{2}, b_{2}] 
$$
and
$$
f_{t_{2}}(s_{1}) = f(s_{1}, t_{2}), 
\text{ for all }s_{1} \in [a_{1}, b_{1}], 
$$
respectively.

\section{The Main Results}

The main results are Theorems \ref{T2.1} and \ref{T2.2}. 
Let us start with a few auxiliary lemmas, which will be formulated for the case of the McShane integral 
($\mathcal{SM}[a,b]$, $\mathcal{S^{*}M}[a,b]$) 
but all of them hold for the Henstock-Kurzweil integral 
($\mathcal{SHK}[a,b]$, $\mathcal{S^{*}HK}[a,b]$) as well. 
It suffices to check their proofs with the
necessary replacement of $\mathcal{M}$-partitions 
by $\mathcal{HK}$-partitions, etc.


\begin{lemma}\label{L2.1}
Let $Z \subset [a,b] = [a_{1}, b_{1}] \times [a_{2}, b_{2}]$. 
Then the following statements hold. 
\begin{itemize}
\item[(i)]
Let $w : [a,b] \to [0, +\infty)$ be such that 
$w(t) >0$ for each $t \in Z$. 
If $w \mathbbm{1}_{Z}$ is McShane integrable on $[a,b]$ with
$$
(M)\int_{[a,b]} w \mathbbm{1}_{Z} =0,
$$
then $\mathbbm{1}_{Z}$ is McShane integrable on $[a,b]$ with
$
(M)\int_{[a,b]} \mathbbm{1}_{Z} =0.
$
\item[(ii)]
Let $f : [a,b] \to X$. 
If $\mathbbm{1}_{Z}$ is McShane integrable on $[a,b]$ with
$$
(M)\int_{[a,b]} \mathbbm{1}_{Z} =0,
$$
then $f \mathbbm{1}_{Z} \in \mathcal{SM}[a,b]$ with
$
(M)\int_{[a,b]} f \mathbbm{1}_{Z} =\theta,
$
where $\theta$ is the zero vector in $X$. 
\end{itemize}
\end{lemma} 
\begin{proof}
$(i)$ 
Let $\varepsilon >0$ be given. 
By Lemma 3.4.2 in \cite{SCH}, for each $n = 0, 1, 2, \dotsc$ 
there exists a gauge $\delta_{n}$ on $[a,b]$ such that 
\begin{equation}\label{eqL21.1}
\left |
\sum_{(t,I) \in P_{n}} w(t) \mathbbm{1}_{Z}(t)|I| - 0 
\right | 
=
\sum_{(t,I) \in P_{n}} w(t) \mathbbm{1}_{Z}(t)|I|
< \frac{\varepsilon}{(n+1)2^{n+1}}
\end{equation}
whenever $P_{n}$ is a $\delta_{n}$-fine $\mathcal{M}$-partition in $[a,b]$.   
We set   
$$
A_{n} = \left \{t \in [a,b] :  \frac{1}{n+1} \leq w(t) < \frac{1}{n} \right \}, 
\text{ for all }n \in \mathbb{N}, 
$$
$A_{0} = \{t \in [a,b] : w(t) \geq 1 \}$ and 
$B_{0} = \{t \in [a,b] : w(t) = 0 \}$. 
Since $B_{0} \cap Z = \emptyset$ it follows that
$\mathbbm{1}_{Z}(t) = 0$ for every $t \in B_{0}$. 
Define a gauge $\delta$ on $[a,b]$ as follows
$$
\delta(t) =
\left \{
\begin{array}{ll}
\delta_{n}(t) & \text{if } t \in A_{n},~ n = 0, 1, 2, \dotsc \\
1 & \text{if } t \in B_{0} \\
\end{array}
\right.
$$
and let $P$ be 
a $\delta$-fine $\mathcal{M}$-partition of $[a,b]$. 
Then we obtain by \eqref{eqL21.1} that
\begin{equation*}
\begin{split}
\left | 
\sum_{(t,I) \in P} \mathbbm{1}_{Z}(t)|I| - 0 
\right |
=&
\sum_{\underset{t \in B_{0}}{(t,I) \in P}} \mathbbm{1}_{Z}(t)|I| + 
\sum_{\underset{t \in A_{0}}{(t,I) \in P}} \mathbbm{1}_{Z}(t)|I| +
\sum_{n=1}^{+\infty} \sum_{\underset{t \in A_{n}}{(t,I) \in P}} \mathbbm{1}_{Z}(t)|I| \\
=&
\sum_{\underset{t \in A_{0}}{(t,I) \in P}} \mathbbm{1}_{Z}(t)|I| +
\sum_{n=1}^{+\infty} (n+1)\sum_{\underset{t \in A_{n}}{(t,I) \in P}} 
\frac{1}{n+1}\mathbbm{1}_{Z}(t)|I| \\
\leq&
\sum_{\underset{t \in A_{0}}{(t,I) \in P}} w(t) \mathbbm{1}_{Z}(t)|I| +
\sum_{n=1}^{+\infty} (n+1)\sum_{\underset{t \in A_{n}}{(t,I) \in P}} 
w(t) \mathbbm{1}_{Z}(t)|I| \\
<&
\frac{\varepsilon}{2} + \sum_{n=1}^{+\infty} \frac{\varepsilon}{2^{n+1}} 
=\frac{\varepsilon}{2} + \frac{\varepsilon}{2} = \varepsilon. 
\end{split}
\end{equation*}
This means that $\mathbbm{1}_{Z}$ is McShane integrable on $[a,b]$ with
$(M)\int_{[a,b]} \mathbbm{1}_{Z} =0$.

$(ii)$  
Let $\varepsilon >0$ be given. 
By Lemma 3.4.2 in \cite{SCH}, 
for each $n \in \mathbb{N}$ there exists a gauge $\delta_{n}$ on $[a,b]$ such that 
\begin{equation}\label{eqL21.2}
\left |
\sum_{(t,I) \in P_{n}} \mathbbm{1}_{Z}(t)|I| - 0 
\right | 
=
\sum_{(t,I) \in P_{n}} \mathbbm{1}_{Z}(t)|I|
< \frac{\varepsilon}{n 2^{n}}
\end{equation}
whenever $P_{n}$ is a $\delta_{n}$-fine $\mathcal{M}$-partition in $[a,b]$.   
We set   
$$
B_{n} = \left \{t \in [a,b] :  n-1 \leq ||f(t)|| < n \right \}, 
\text{ for all }n \in \mathbb{N}. 
$$
Define a gauge $\delta$ on $[a,b]$ so that
$$
\delta(t) =
\delta_{n}(t)
$$
whenever $n \in \mathbb{N}$ and $t \in B_{n}$. 
If $P$ is  
a $\delta$-fine $\mathcal{M}$-partition of $[a,b]$,  
then we obtain by \eqref{eqL21.2} that
\begin{equation*} 
\begin{split}
\sum_{(t,I) \in P} 
\left \| 
f(t) \mathbbm{1}_{Z}(t)|I| - \theta 
\right \|
\leq&
\sum_{n=1}^{+\infty} 
\sum_{\underset{t \in B_{n}}{(t,I) \in P}} ||f(t)|| \mathbbm{1}_{Z}(t)|I|
\\
<&
\sum_{n=1}^{+\infty} 
n\sum_{\underset{t \in B_{n}}{(t,I) \in P}} \mathbbm{1}_{Z}(t)|I| 
<
\sum_{n=1}^{+\infty} \frac{\varepsilon}{2^{n}} = \varepsilon. 
\end{split}
\end{equation*}
This means that $f \mathbbm{1}_{Z}$ is strongly McShane integrable on $[a,b]$ with
$(M)\int_{[a,b]}f  \mathbbm{1}_{Z} =\theta$ and the proof is finished.
\end{proof}


\begin{lemma}\label{L2.2}
Assume that $f \in \mathcal{SM}[a,b]$ and  
for each $t_{1} \in [a_{1}, b_{1}]$ and $t_{2} \in [a_{2}, b_{2}]$, 
we have
$$
f_{t_{1}} \in \mathcal{SM}[a_{2}, b_{2}]
\text{ and }
f_{t_{2}} \in \mathcal{SM}[a_{1}, b_{1}].
$$  
Then the following statements hold. 
\begin{itemize}
\item[(i)]
The function 
$g(t_{1}) = (M)\int_{[a_{2}, b_{2}]} f_{t_{1}}$, 
for all $t_{1} \in [a_{1}, b_{1}]$,  
is strongly McShane integrable on $[a_{1}, b_{1}]$ and 
$$(M)\int_{[a_{1}, b_{1}]} \left ( (M)\int_{[a_{2}, b_{2}]} f_{t_{1}} \right )
=(M)\int_{[a, b]} f.
$$
\item[(ii)]
The function 
$h(t_{2}) = (M)\int_{[a_{1}, b_{1}]} f_{t_{2}}$, 
for all $t_{2} \in [a_{2}, b_{2}]$,  
is strongly McShane integrable on $[a_{2}, b_{2}]$ and 
$$(M)\int_{[a_{2}, b_{2}]} \left ( (M)\int_{[a_{1}, b_{1}]} f_{t_{2}} \right )
=(M)\int_{[a, b]} f.
$$
\end{itemize}
\end{lemma}
\begin{proof}
Let $\varepsilon >0$ be given. 
Then, since  $f$ is strongly McShane integrable on 
$[a,b]$ there exists a gauge $\delta$ on $[a,b]$ such that 
\begin{equation}\label{eqL22.1}
\begin{split}
\sum_{(t, I) \in P} 
\left \| f(t)|I| - (M)\int_{I} f \right \|
< \frac{\varepsilon}{2}
\end{split}
\end{equation}
for each $\delta$-fine $\mathcal{M}$-partition $P$ of $[a,b]$.

Since $f_{t_{1}}$ is  strongly McShane integrable 
on $[a_{2}, b_{2}]$ whenever $t_{1} \in [a_{1}, b_{1}]$, 
there exists a gauge $\delta^{(t_{1})}_{2}$ on $[a_{2}, b_{2}]$ 
such that  
\begin{equation}\label{eqL22.2}
\begin{split}
\sum_{(t_{2}, I_{2}) \in Q^{(t_{1})}_{2}} 
\left \| f_{t_{1}}(t_{2})|I_{2}| - (M)\int_{I_{2}} f_{t_{1}} \right \|
< \frac{\varepsilon}{2(1+(b_{1} - a_{1}))}
\end{split}
\end{equation}
for each $\delta^{(t_{1})}_{2}$-fine $\mathcal{M}$-partition $Q^{(t_{1})}_{2}$ of $[a_{2}, b_{2}]$. 
We can choose each $\delta^{(t_{1})}_{2}$ so that 
$$
\delta^{(t_{1})}_{2}(t_{2}) \leq \delta(t_{1}, t_{2}), 
\text{ for all }t_{2} \in [a_{2}, b_{2}].
$$

We now define a gauge $\delta_{1}$ on $[a_{1}, b_{1}]$ by setting 
$$
\delta_{1}(t_{1}) = \min 
\left \{ 
\delta^{(t_{1})}_{2}(t_{2}) : 
(t_{2}, I_{2}) \in  Q^{(t_{1})}_{2} 
\right \}, 
\text{ for all }t_{1} \in [a_{1}, b_{1}]
$$
and let $Q_{1}$ be 
a $\delta_{1}$-fine $\mathcal{M}$-partition of $[a_{1}, b_{1}]$.  
Then
$$
Q = 
\left \{
((t_{1}, t_{2}), I_{1} \times I_{2}) :
(t_{1}, I_{1}) \in Q_{1}
\text{ and }
(t_{2}, I_{2}) \in Q^{(t_{1})}_{2}
\right \}
$$
is a $\delta$-fine $\mathcal{M}$-partition of $[a, b]$.  
We have
\begin{equation*}
\begin{split}
&\sum_{(t_{1}, I_{1}) \in Q_{1}} 
\left \| g(t_{1})|I_{1}| - (M)\int_{I_{1} \times [a_{2}, b_{2}]} f \right \| \leq \\
\leq 
&\sum_{(t_{1}, I_{1}) \in Q_{1}} 
\left ( |I_{1}| \left \| (M)\int_{[a_{2},b_{2}]} f_{t_{1}} - \sum_{(t_{2}, I_{2}) \in  Q^{(t_{1})}_{2}}  f_{t_{1}}(t_{2}) |I_{2}| \right \| 
+
\left\| \sum_{(t_{2}, I_{2}) \in  Q^{(t_{1})}_{2}}  f(t_{1},t_{2})|I_{1}|.|I_{2}|  - 
(M)\int_{I_{1} \times [a_{2}, b_{2}]} f  
\right \| \right )
\end{split}
\end{equation*}
and since
\begin{equation*}
\begin{split}
&\sum_{(t_{1}, I_{1}) \in Q_{1}} 
\left ( |I_{1}| \left \| (M)\int_{[a_{2},b_{2}]} f_{t_{1}} - \sum_{(t_{2}, I_{2}) \in  Q^{(t_{1})}_{2}}  f_{t_{1}}(t_{2}) |I_{2}| \right \| 
+
\left\| \sum_{(t_{2}, I_{2}) \in  Q^{(t_{1})}_{2}}  f(t_{1},t_{2})|I_{1}|.|I_{2}|  - 
(M)\int_{I_{1} \times [a_{2}, b_{2}]} f  
\right \| \right )\\
=&
\sum_{(t_{1}, I_{1}) \in Q_{1}} 
\left ( |I_{1}| 
\left \| \sum_{(t_{2}, I_{2}) \in  Q^{(t_{1})}_{2}}  
\left ( f_{t_{1}}(t_{2})|I_{2}| - (M)\int_{I_{2}}f_{t_{1}} \right )\right \| 
+
\left\|
\sum_{(t_{2}, I_{2}) \in  Q^{(t_{1})}_{2}}  
\left ( f(t_{1},t_{2})|I_{1}|.|I_{2}|  - (M)\int_{I_{1} \times I_{2}} f \right )  
\right \| \right )\\
\leq&
\sum_{(t_{1}, I_{1}) \in Q_{1}} 
|I_{1}| 
\sum_{(t_{2}, I_{2}) \in  Q^{(t_{1})}_{2}}  
\left \| f_{t_{1}}(t_{2})|I_{2}| - (M)\int_{I_{2}}f_{t_{1}} \right \| 
+
\sum_{(t_{1}, I_{1}) \in Q_{1}}  \sum_{(t_{2}, I_{2}) \in  Q^{(t_{1})}_{2}}  
\left \| f(t_{1},t_{2})|I_{1}|.|I_{2}|  - (M)\int_{I_{1} \times I_{2}} f \right \|
\end{split}
\end{equation*}
we obtain by \eqref{eqL22.1} and \eqref{eqL22.2} that
\begin{equation*}
\begin{split}
\sum_{(t_{1}, I_{1}) \in Q_{1}} 
\left \| g(t_{1})|I_{1}| - (M)\int_{I_{1} \times [a_{2}, b_{2}]} f \right \| 
<
\sum_{(t_{1}, I_{1}) \in Q_{1}} |I_{1}| 
\frac{\varepsilon}{2(1+(b_{1} - a_{1}))} + \frac{\varepsilon}{2}
< \frac{\varepsilon}{2} +  \frac{\varepsilon}{2} = \varepsilon.
\end{split}
\end{equation*}
This means that $g$ is strongly McShane integrable 
on $[a_{1}, b_{1}]$ and $(i)$ holds. 

Since the proof of $(ii)$ is similar to that of $(i)$, 
the lemma is proved.
\end{proof}


\begin{lemma}\label{L2.3} 
Let $f \in \mathcal{S^{*}M}[a,b]$, let 
$$
Z_{1} = 
\left \{
t_{1} \in [a_{1},b_{1}]: f_{t_{1}} \not\in \mathcal{S^{*}M}[a_{2}, b_{2}],  
\right \}, \quad
Z_{2} = 
\left \{
t_{2} \in [a_{2},b_{2}]: f_{t_{2}} \not\in \mathcal{S^{*}M}[a_{1}, b_{1}]  
\right \}
$$ 
and $Z = (Z_{1} \times [a_{2},b_{2}]) \cup  ([a_{1},b_{1}] \times Z_{2})$.

Then the following statements hold. 
\begin{itemize}
\item[(i)]
$\mathbbm{1}_{Z_{1}}$ is Mcshane integrable 
on $[a_{1}, b_{1}]$ with 
$(M)\int_{[a_{1}, b_{1}]} \mathbbm{1}_{Z_{1}} = 0$. 
\item[(ii)]
$\mathbbm{1}_{Z_{2}}$ is Mcshane integrable 
on $[a_{2}, b_{2}]$ with 
$(M)\int_{[a_{2}, b_{2}]} \mathbbm{1}_{Z_{2}} = 0$.
\item[(iii)]
$\mathbbm{1}_{Z}$ is Mcshane integrable  
on $[a, b]$ with 
$(M)\int_{[a, b]} \mathbbm{1}_{Z} = 0$.
\end{itemize}
\end{lemma} 
\begin{proof}
$(i)$ 
By virtue of Theorem 3.6.13 in \cite{SCH}, for each $t_{1} \in Z_{1}$ 
there exists $w(t_{1})>0$ with the following property: 
for each gauge $\delta^{(t_{1})}_{2}$ on $[a_{2}, b_{2}]$ 
there exist a pair of $\delta^{(t_{1})}_{2}$-fine $\mathcal{M}$-partitions 
$Q^{(t_{1})}_{1}, Q^{(t_{1})}_{2}$ of $[a_{2}, b_{2}]$ such that 
\begin{equation}\label{eqL24.1}
\begin{split} 
\sum_{(t_{2}, I_{2}) \in Q^{(t_{1})}_{1} } 
\sum_{(s_{2}, J_{2}) \in Q^{(t_{1})}_{2} }
\| f_{t_{1}}(t_{2}) - f_{t_{1}}(s_{2}) \| . |I_{2} \cap J_{2}|
\geq 
w(t_{1}).
\end{split}
\end{equation}
If we choose $w(t_{1}) = 0$ at all $t_{1} \in [a_{1}, b_{1}] \setminus Z_{1}$, 
then $w = w \mathbbm{1}_{Z_{1}}$. 
Thus, if $w$ is McShane integrable 
on $[a_{1}, b_{1}]$ with 
\begin{equation*}
(M)\int_{[a_{1}, b_{1}]} w = 0, 
\end{equation*}
then we obtain by $(i)$ in Lemma \ref{L2.1} that $(i)$ holds.

Since $f \in \mathcal{S^{*}M}[a,b]$,  
given $\varepsilon >0$ there exists a gauge $\Delta$ on $[a,b]$ such that 
\begin{equation}\label{eqL24.2}
\begin{split}
\sum_{(t,I) \in Q_{1} } \sum_{(s,J) \in Q_{2} } 
\| f(t) - f(s)\|.|I \cap J|
< \varepsilon
\end{split}
\end{equation}
for each pair of 
$\Delta$-fine $\mathcal{M}$-partitions $Q_{1}, Q_{2}$ of $[a,b]$. 
 
Note that for each $t_{1} \in [a_{1}, b_{1}]$ 
the function $\Delta^{(t_{1})}_{2} : [a_{2},b_{2}] \to (0, +\infty)$ 
defined by
$$
\Delta^{(t_{1})}_{2}(t_{2}) = \Delta(t_{1}, t_{2}), 
\text{ for all }t_{2} \in [a_{2}, b_{2}]
$$ 
is a gauge on $[a_{2}, b_{2}]$. 
Then, by \eqref{eqL24.1} for each $t_{1} \in Z_{1}$, we can choose 
a pair of $\Delta^{(t_{1})}_{2}$-fine $\mathcal{M}$-partitions 
$P^{(t_{1})}_{1}, P^{(t_{1})}_{2}$ of $[a_{2}, b_{2}]$ such that 
\begin{equation}\label{eqL24.3}
\begin{split}
\sum_{(t_{2}, I_{2}) \in P^{(t_{1})}_{1} } 
\sum_{(s_{2}, J_{2}) \in P^{(t_{1})}_{2} }
\| f_{t_{1}}(t_{2}) - f_{t_{1}}(s_{2}) \| . |I_{2} \cap J_{2}|
\geq 
w(t_{1}).
\end{split}
\end{equation}
For each $t_{1} \in [a_{1}, b_{1}] \setminus Z_{1}$, 
we choose a $\Delta^{(t_{1})}_{2}$-fine $\mathcal{M}$-partition 
$P^{(t_{1})}$ of $[a_{2}, b_{2}]$ and set  
$P^{(t_{1})}_{1} = P^{(t_{1})}_{2}=P^{(t_{1})}$. 
In this case, it easy to see that \eqref{eqL24.3} holds also.

We now define a gauge $\Delta_{1}$ on $[a_{1}, b_{1}]$ by setting 
$$
\Delta_{1}(t_{1}) = \min 
\left \{
\Delta(t_{1}, t_{2}) : (t_{2}, I_{2}) \in P^{(t_{1})}_{1} \cup P^{(t_{1})}_{2}
\right \},
\text{ for all }t_{1} \in [a_{1}, b_{1}],
$$ 
and let $\pi$ be a 
$\Delta_{1}$-fine $\mathcal{M}$-partition of $[a_{1}, b_{1}]$. 
Since
$$
P_{1} = 
\{
((t_{1},t_{2}), I_{1} \times I_{2}) : 
(t_{1}, I_{1}) \in \pi \text{ and } (t_{2}, I_{2}) \in P^{(t_{1})}_{1}  
\}
$$
and
$$
P_{2} = 
\{
((t_{1},s_{2}), I_{1} \times J_{2}) : 
(t_{1}, I_{1}) \in \pi \text{ and } (s_{2}, J_{2}) \in P^{(t_{1})}_{2}  
\}
$$
are $\Delta$-fine $\mathcal{M}$-partitions of $[a,b]$, 
we obtain by \eqref{eqL24.3} and \eqref{eqL24.2} that
\begin{equation*}
\begin{split}
&\left | \sum_{(t_{1}, I_{1}) \in \pi} w(t_{1})|I_{1}| - 0 \right | 
= \sum_{(t_{1}, I_{1}) \in \pi} w(t_{1})|I_{1}| \\
\leq&
\sum_{(t_{1}, I_{1}) \in \pi} |I_{1}|
\sum_{(t_{2}, I_{2}) \in P^{(t_{1})}_{1} } 
\sum_{(s_{2}, J_{2}) \in P^{(t_{1})}_{2} }
\| f_{t_{1}}(t_{2}) - f_{t_{1}}(s_{2}) \| . |I_{2} \cap J_{2}| \\
=&
\sum_{(t_{1}, I_{1}) \in \pi} 
\sum_{(t_{2}, I_{2}) \in P^{(t_{1})}_{1} } 
\sum_{(s_{2}, J_{2}) \in P^{(t_{1})}_{2} }
\| f_{t_{1}}(t_{2}) - f_{t_{1}}(s_{2}) \| . 
|(I_{1} \times I_{2}) \cap (I_{1} \times J_{2})|\\
=&
\sum_{\underset{(t_{1}, I_{1}) \in \pi}{(t_{2}, I_{2}) \in P^{(t_{1})}_{1}}} 
\sum_{\underset{(t_{1}, I_{1}) \in \pi}{(s_{2}, J_{2}) \in P^{(t_{1})}_{2}}}
\| f(t_{1}, t_{2}) - f(t_{1}, s_{2}) \| . 
|(I_{1} \times I_{2}) \cap (I_{1} \times J_{2})|\\
=&
\sum_{((t_{1},t_{2}), I_{1} \times I_{2}) \in P_{1}} 
\sum_{((t_{1},s_{2}), I_{1} \times J_{2}) \in P_{2}}
\| f(t_{1}, t_{2}) - f(t_{1}, s_{2}) \| . 
|(I_{1} \times I_{2}) \cap (I_{1} \times J_{2})| < \varepsilon.
\end{split}
\end{equation*} 
This means that $w$ is McShane integrable on $[a_{1},b_{1}]$ 
with $(M)\int_{ [a_{1},b_{1}]}w =0$.

The proof of $(ii)$ is similar to that of $(i)$.

$(iii)$ 
Since $\mathbbm{1}_{Z_{1}}$ is Mcshane integrable 
on $[a_{1}, b_{1}]$ with 
$(M)\int_{[a_{1}, b_{1}]} \mathbbm{1}_{Z_{1}} = 0$,
given $\varepsilon >0$ there exists a gauge $\delta_{1}$ on $[a_{1},b_{1}]$ such that 
\begin{equation}\label{eqL24.4}
\sum_{(t_{1}, I_{1}) \in \pi} \mathbbm{1}_{Z_{1}}(t_{1})|I_{1}| 
< \frac{\varepsilon}{1 + (b_{2} - a_{2})}
\end{equation}
for each $\delta_{1}$-fine $\mathcal{M}$-partition $\pi$ of $[a_{1},b_{1}]$.

We  now define a gauge $\delta$ on $[a,b]$ by setting 
$$
\delta(t_{1}, t_{2}) = \delta(t_{1}), 
\text{ for all }
(t_{1}, t_{2}) \in [a,b]= [a_{1}, b_{1}] \times [a_{2}, b_{2}]
$$
and let $P$ be a $\delta$-fine $\mathcal{M}$-partition of $[a,b]$. 
There exists a finite collection 
$\mathcal{D}_{P}$ of pairwise non-overlapping intervals in 
$[a_{2},b_{2}]$ such that
\begin{itemize}
\item
$[a_{2},b_{2}] = \bigcup_{I \in \mathcal{D}_{P}} I$,
\item
for each $I \in \mathcal{D}_{P}$ there exists 
$P^{(I)} \subset P$ such that 
$$
P^{(I)} = \{ I_{2} \in \mathcal{I}_{[a_{2},b_{2}]} : ((t_{1},t_{2}), I_{1} \times I_{2}) \in P
\text{ and }
I \subset I_{2} \}
$$
and
$$
I = \bigcap_{((t_{1},t_{2}), I_{1} \times I_{2}) \in P^{(I)}} I_{2}.
$$
\end{itemize}
Hence, for each $I \in \mathcal{D}_{P}$ the collection
$$
P^{(I)}_{1} = 
\{ (t_{1}, I_{1}): ((t_{1},t_{2}), I_{1} \times I_{2}) \in P^{(I)} \}
$$
is a $\delta_{1}$-fine $\mathcal{M}$-partition of $[a_{1},b_{1}]$. 
Note that
\begin{equation*}
\begin{split}
&\sum_{((t_{1},t_{2}), I_{1} \times I_{2}) \in P} \mathbbm{1}_{Z_{1}}(t_{1}) |I_{1}| . |I_{2}| 
=
\sum_{I \in \mathcal{D}_{P}} 
\left (
\sum_{((t_{1},t_{2}), I_{1} \times I_{2}) \in P^{(I)}} \mathbbm{1}_{Z_{1}}(t_{1}) |I_{1}| . |I \cap I_{2}|
\right ) \\
=&
\sum_{I \in \mathcal{D}_{P}} |I|
\left (
\sum_{((t_{1},t_{2}), I_{1} \times I_{2}) \in P^{(I)}} \mathbbm{1}_{Z_{1}}(t_{1}) |I_{1}| 
\right ) 
=
\sum_{I \in \mathcal{D}_{P}} |I|
\left (
\sum_{(t_{1}, I_{1}) \in P^{(I)}_{1}} \mathbbm{1}_{Z_{1}}(t_{1}) |I_{1}| 
\right ).
\end{split}
\end{equation*} 
and
$$
\mathbbm{1}_{Z}(t_{1}, t_{2}) = 
\mathbbm{1}_{Z_{1}}(t_{1}). \mathbbm{1}_{Z_{2}}(t_{2}),
\text{ for all }(t_{1}, t_{2}) \in [a,b] = [a_{1},b_{1}] \times [a_{2},b_{2}].
$$
Therefore, we obtain by \eqref{eqL24.4} that
\begin{equation*}
\begin{split}
\left | 
\sum_{((t_{1},t_{2}), I_{1} \times I_{2}) \in P} 
\mathbbm{1}_{Z}(t_{1},t_{2}) |I_{1}\times I_{2}| - 0
\right | 
=&
\sum_{((t_{1},t_{2}), I_{1} \times I_{2}) \in P} 
\mathbbm{1}_{Z_{1}}(t_{1}). \mathbbm{1}_{Z_{2}}(t_{2}) |I_{1}|.|I_{2}| \\
\leq&
\sum_{((t_{1},t_{2}), I_{1} \times I_{2}) \in P} 
\mathbbm{1}_{Z_{1}}(t_{1}) |I_{1}|.|I_{2}| \\
<&
\frac{\varepsilon}{1 + (b_{2} - a_{2})}\sum_{I \in \mathcal{D}_{P}} 
|I| 
= \frac{\varepsilon}{1 + (b_{2} - a_{2})}(b_{2} - a_{2})<\varepsilon.
\end{split}
\end{equation*} 
This means that $(iii)$ holds and the proof is finished.
\end{proof}

We are now ready to present the first main result.

\begin{theorem}\label{T2.1}
Let $f \in \mathcal{SM}[a,b]$, let 
$$
Z_{1} = 
\left \{
t_{1} \in [a_{1},b_{1}]: f_{t_{1}} \not\in \mathcal{SM}[a_{2}, b_{2}] 
\right \},
\quad
Z_{2} = 
\left \{
t_{2} \in [a_{2},b_{2}]: f_{t_{2}} \not\in \mathcal{SM}[a_{1}, b_{1}]  
\right \},
$$
$Z = (Z_{1} \times [a_{2},b_{2}]) \cup  ([a_{1},b_{1}] \times Z_{2})$ and 
$f_{0} = f . \mathbbm{1}_{[a,b] \setminus Z}$. 

Then the following statements hold.
\begin{itemize}
\item[(i)]
$f_{0} \in \mathcal{SM}[a,b]$ and for each $(t_{1},t_{2}) \in [a_{1}, b_{1}] \times  [a_{2}, b_{2}]$, we have 
$(f_{0})_{t_{1}} \in \mathcal{SM}[a_{2},b_{2}]$ and 
$(f_{0})_{t_{2}} \in \mathcal{SM}[a_{1},b_{1}]$.
\item[(ii)]
The function 
$$
t_{1} \to g(t_{1}) = (M)\int_{[a_{2}, b_{2}]} (f_{0})_{t_{1}},\text{ for all }t_{1} \in [a_{1}, b_{1}],
$$  
is strongly McShane integrable on $[a_{1}, b_{1}]$ and 
\begin{equation*}
\begin{split}
(M)\int_{[a_{1}, b_{1}]} \left ( (M)\int_{[a_{2}, b_{2}]} (f_{0})_{t_{1}} \right )
=(M)\int_{[a, b]} f.
\end{split}
\end{equation*} 
\item[(iii)]
The function 
$$
t_{2} \to h(t_{2}) = (M)\int_{[a_{1}, b_{1}]} (f_{0})_{t_{2}}, \text{ for all }t_{2} \in [a_{2}, b_{2}],
$$  
is strongly McShane integrable on $[a_{2}, b_{2}]$ and 
\begin{equation*}
\begin{split}
(M)\int_{[a_{2}, b_{2}]} \left ( (M)\int_{[a_{1}, b_{1}]} (f_{0})_{t_{2}} \right )
=(M)\int_{[a, b]} f.
\end{split}
\end{equation*}
\end{itemize}
\end{theorem}
\begin{proof}
$(i)$  
Since $\mathcal{SM}[a,b] = \mathcal{S^{*}M}[a,b]$,  
we obtain by Lemma \ref{L2.3} that the function $\mathbbm{1}_{Z}$ 
is McShane integrable with $(M)\int_{[a,b]}\mathbbm{1}_{Z} =0$. 
Hence,  by $(ii)$ in Lemma \ref{L2.1} we have 
$f \mathbbm{1}_{Z} \in \mathcal{SM}[a,b]$ 
and 
$(M)\int_{[a,b]}f \mathbbm{1}_{Z} = \theta$. 
It follows that 
\begin{equation}\label{eqT21.1}
\begin{split}
f - f \mathbbm{1}_{Z} = f_{0} \in \mathcal{SM}[a,b]
\quad\text{and}\quad
(M)\int_{[a, b]} f_{0} 
=&
(M)\int_{[a, b]} f. 
\end{split}
\end{equation}
We now fix an arbitrary $t_{1} \in [a_{1}, b_{1}]$. 
There are two cases to consider. 
\begin{itemize}
\item[(a)]
$t_{1} \in Z_{1}$. 
In this case, we have $(f_{0})_{t_{1}}(t_{2}) = \theta$ at all $t_{2} \in [a_{2},b_{2}]$. 
Thus,   $(f_{0})_{t_{1}} \in \mathcal{SM}[a_{2},b_{2}]$. 
\item[(b)]
$t_{1} \not\in Z_{1}$. 
In this case, we have $f_{t_{1}} \in \mathcal{SM}[a,b]$ and 
$(f_{0})_{t_{1}}(t_{2}) = f_{t_{1}}(t_{2})-f_{t_{1}}(t_{2})\mathbbm{1}_{Z_{2}}(t_{2})$ at all $t_{2} \in [a_{2},b_{2}]$. 
Lemma \ref{L2.3} together with Lemma \ref{L2.1} yields that 
$f_{t_{1}}.\mathbbm{1}_{Z_{2}} \in \mathcal{SM}[a_{2},b_{2}]$ 
with $(M)\int_{[a_{2},b_{2}]} f_{t_{1}}.\mathbbm{1}_{Z_{2}} = \theta$. 
Therefore, $(f_{0})_{t_{1}} \in \mathcal{SM}[a_{2},b_{2}]$.
\end{itemize}
Hence, we have $(f_{0})_{t_{1}} \in \mathcal{SM}[a_{2},b_{2}]$ 
for each $t_{1} \in [a_{1},b_{1}]$. 
Similarly, it can be proved that 
$(f_{0})_{t_{2}} \in \mathcal{SM}[a_{1},b_{1}]$ 
for each $t_{2} \in [a_{2},b_{2}]$.

Therefore, Lemma \ref{L2.2} together with $(i)$ and \eqref{eqT21.1} yields that 
$(ii)$ and $(iii)$ hold, and this ends the proof. 
\end{proof}

\begin{theorem}\label{T2.2}
Let $f \in \mathcal{S^{*}HK}[a,b]$, let 
$$
Z_{1} = 
\left \{
t_{1} \in [a_{1},b_{1}]: f_{t_{1}} \not\in \mathcal{S^{*}HK}[a_{2}, b_{2}] 
\right \},
\quad
Z_{2} = 
\left \{
t_{2} \in [a_{2},b_{2}]: f_{t_{2}} \not\in \mathcal{S^{*}HK}[a_{1}, b_{1}]  
\right \},
$$
$Z = (Z_{1} \times [a_{2},b_{2}]) \cup  ([a_{1},b_{1}] \times Z_{2})$ and 
$f_{0} = f \mathbbm{1}_{[a,b] \setminus Z}$. 

Then the following statements hold.
\begin{itemize}
\item[(i)]
$f_{0} \in \mathcal{SHK}[a,b]$ and for each $(t_{1},t_{2}) \in [a_{1}, b_{1}] \times  [a_{2}, b_{2}]$, we have 
$(f_{0})_{t_{1}} \in \mathcal{SHK}[a_{2},b_{2}]$ and 
$(f_{0})_{t_{2}} \in \mathcal{SHK}[a_{1},b_{1}]$.
\item[(ii)]
The function 
$$
t_{1} \to g(t_{1}) = (HK)\int_{[a_{2}, b_{2}]} (f_{0})_{t_{1}},\text{ for all }t_{1} \in [a_{1}, b_{1}],
$$  
is strongly Henstock-Kurweil integrable on $[a_{1}, b_{1}]$ and 
\begin{equation*}
\begin{split}
(HK)\int_{[a_{1}, b_{1}]} \left ( (HK)\int_{[a_{2}, b_{2}]} (f_{0})_{t_{1}} \right )
=(HK)\int_{[a, b]} f.
\end{split}
\end{equation*} 
\item[(iii)]
The function 
$$
t_{2} \to h(t_{2}) = (HK)\int_{[a_{1}, b_{1}]} (f_{0})_{t_{2}}, 
\text{ for all }t_{2} \in [a_{2}, b_{2}],
$$  
is strongly Henstock-Kurweil integrable on $[a_{2}, b_{2}]$ and 
\begin{equation*}
\begin{split}
(HK)\int_{[a_{2}, b_{2}]} \left ( (HK)\int_{[a_{1}, b_{1}]} (f_{0})_{t_{2}} \right )
=(HK)\int_{[a, b]} f.
\end{split}
\end{equation*}
\end{itemize}
\end{theorem}
\begin{proof}
$(i)$  
By Lemma \ref{L2.3} the function $\mathbbm{1}_{Z}$ 
is Henstock-Kurweil integrable with $(HK)\int_{[a,b]}\mathbbm{1}_{Z} =0$. 
Hence,  by $(ii)$ in Lemma \ref{L2.1} we have 
$f \mathbbm{1}_{Z} \in \mathcal{SHK}[a,b]$ 
and 
$(HK)\int_{[a,b]}f \mathbbm{1}_{Z} = \theta$. 
It follows that 
\begin{equation}\label{eqT22.1}
\begin{split}
f - f \mathbbm{1}_{Z} = f_{0} \in \mathcal{SHK}[a,b]
\quad\text{and}\quad
(HK)\int_{[a, b]} f_{0} 
=&
(HK)\int_{[a, b]} f. 
\end{split}
\end{equation}
We now fix an arbitrary $t_{1} \in [a_{1}, b_{1}]$. 
There are two cases to consider. 
\begin{itemize}
\item[(a)]
$t_{1} \in Z_{1}$. 
In this case, we have $(f_{0})_{t_{1}}(t_{2}) = \theta$ at all $t_{2} \in [a_{2},b_{2}]$. 
Thus,   $(f_{0})_{t_{1}} \in \mathcal{SHK}[a_{2},b_{2}]$.
\item[(b)]
$t_{1} \not\in Z_{1}$. 
In this case, we have $f_{t_{1}} \in \mathcal{SHK}[a,b]$ and 
$(f_{0})_{t_{1}}(t_{2}) = f_{t_{1}}(t_{2})-f_{t_{1}}(t_{2})\mathbbm{1}_{Z_{2}}(t_{2})$ at all $t_{2} \in [a_{2},b_{2}]$. 
Lemma \ref{L2.3} together with Lemma \ref{L2.1} yields that 
$f_{t_{1}}.\mathbbm{1}_{Z_{2}} \in \mathcal{SHK}[a_{2},b_{2}]$ 
with $(HK)\int_{[a_{2},b_{2}]} f_{t_{1}}.\mathbbm{1}_{Z_{2}} = \theta$. 
Therefore, $(f_{0})_{t_{1}} \in \mathcal{SHK}[a_{2},b_{2}]$.
\end{itemize}
Hence, we have $(f_{0})_{t_{1}} \in \mathcal{SHK}[a_{2},b_{2}]$ 
for each $t_{1} \in [a_{1},b_{1}]$. 
Similarly, it can be proved that 
$(f_{0})_{t_{2}} \in \mathcal{SHK}[a_{1},b_{1}]$ 
for each $t_{2} \in [a_{2},b_{2}]$.

Therefore, Lemma \ref{L2.2} together with $(i)$ and \eqref{eqT22.1} yields that 
$(ii)$ and $(iii)$ hold, and this ends the proof. 
\end{proof}

\bibliographystyle{plain}

\end{document}